\documentclass[10pt]{amsart}
\usepackage{amsmath}
\usepackage{amssymb}
\usepackage{graphicx}

\def\g{\gamma}

\def\E{\textsf{E}}
\def\cS{{\mathcal S}}
\def\P{\textsf{P}}
\def\D{{\Delta}}

\def\G{{\Gamma}}

\def\b{\beta}
\def\a{\alpha} 
\newtheorem{theorem}{Theorem}
\newtheorem{lemma}[theorem]{Lemma}
\newtheorem{corollary}[theorem]{Corollary}
\newtheorem{proposition}[theorem]{Proposition}

\begin{document}

\title[Greek letters in staircase tableaux]{Greek letters in  random  staircase tableaux{*}}

 \author[Sandrine Dasse--Hartaut]{Sandrine Dasse--Hartaut$\dagger$}\thanks{$\dagger$ LIAFA,
Universit\'e Paris Diderot--Paris 7,
F-75205 Paris, France, dasse@liafa.jussieu.fr} \author[Pawe{\l} Hitczenko]{Pawe{\l} Hitczenko$\ddagger$}\thanks{$\ddagger$ Institute of Mathematics, Polish Academy of Sciences, ul.~\'Sniadeckich~8,  00-956 Warszawa, Poland and  Department of Mathematics and Information Science, Warsaw University of Technology, Pl. Politechniki~1, 00-661 Warszawa, Poland, phitczenko@math.drexel.edu}
\thanks{{*} The work of the first author was carried out while she  held an ANR Gamma internship at LIPN, Universit\'e Paris Nord, under the direction of Fr\'ed\'erique Bassino (LIPN)  and Sylvie Corteel (LIAFA). She would like to thank both of them for their guidance,   the members of LIPN for their hospitality, and ANR Gamma for the support.  
The second author was partially supported  by NSA Grant \#H98230-09-1-0062. Most of his work was done during his stay at LIPN in July 2010 and completed while he was at the Institute of Mathematics of the Polish Academy of Sciences and the Technical University of Warsaw in the Fall of 2010. 
He would like to thank Fr\'ed\'erique Bassino  for the invitation and  acknowledge the hospitality of these institutions}

\maketitle

\begin{abstract}
In this paper we study a relatively new combinatorial object called staircase tableaux. 
Staircase tableaux were introduced by Corteel and Williams in the connection with Asymmetric Exclusion Process and has since found interesting connections with Askey--Wilson polynomials. We develop a probabilistic approach that allows us to analyze several parameters of a randomly chosen staircase tableau of a given size. In particular, we obtain limiting distributions for  statistics associated with appearances of Greek letters in staircase tableaux.
 \end{abstract}

\vspace{.8cm}

\noindent {\em Key words and phrases:} asymmetric exclusion process, asymptotic normality, staircase tableaux.

\vspace{.8cm}  

\section{introduction}
An interesting combinatorial structure, called staircase tableaux, was introduced in  recent work of Corteel and Williams \cite{cw_pnas, cw}. 
Staircase tableaux are related to the asymmetric exclusion process on an one-dimensional lattice with open boundaries, the ASEP. This is an important and heavily studied particle model in statistical mechanics (we refer to \cite{cw} for some background information on several versions of that model and their  applications and connections to other branches of science). The study of the generating function of the staircase tableau has given a combinatorial  formula for the steady state probability of the ASEP.  Explicit expressions for the steady state probabilities were first given in \cite{Derrida1}. In their work  \cite{cw,cw_pnas} Corteel and Williams used staircase tableaux  to give a combinatorial formula for the moments of the (weight function
of the) Askey-Wilson polynomials; for a follow--up work see \cite{cssw}.

The authors of \cite{cw} called  for further investigation of the staircase tableaux because of their combinatorial interest and their potential connection to geometry.  In this paper we take up that issue and  study some basic properties of staircase tableaux. More precisely, we analyze the distribution of various parameters associated with appearances of Greek  letters $\a$, $\b$, $\delta$, and $\g$ in randomly chosen staircase tableau of size $n$ (see the next section or, e.g. \cite[Section~2]{cw},  for the definitions and  the meaning of these symbols).

Staircase tableaux are generalizations of permutation tableaux (see e.g. \cite{ch,CW1,CW2,hj} and references therein for more information on these objects and their connection to a version of ASEP referred to as the partially asymmetric exclusion process; PASEP). For permutation  tableaux, the authors of \cite{ch} developed a probabilistic approach that later allowed the derivation of the limiting (and even exact) distributions of various parameters  of the permutation tableaux. Our goal here is the same: in Section~\ref{setup} we  develop a probabilistic approach parallel to that of \cite{ch} that  allows us to 
compute generating functions of various quantities associated with staircase tableaux (see Corollary~\ref{cor:gfr} in Section~\ref{count} and Proposition~\ref{prop:gfr} in Subsection~\ref{subsec:euler}).  As a consequence, we obtain the exact or limiting distributions of the parameters we study.
  
 Our main results are gathered in Section~\ref{sec:mainres} and may be summarized by stating that the five parameters we consider have asymptotically normal distribution when normalized in a usual way (i.e. centered by the mean and scaled by the square root of the variance). Thus, for example, Theorem~\ref{thm:diag} asserts that the number $A_n$ of $\alpha$ or $\gamma$ on the diagonal of a randomly chosen staircase tableau of size $n$ has expected value $n/2$, variance $(n+1)/12$ and that  
 $(A_n-n/2)/\sqrt{n/12}$ converges in distribution to the standard normal random variable. We refer to  Theorems~\ref{thm:ent} and~\ref{thm:diag} for precise statements. 
 
In Section~\ref{proofs} below (see comments at the beginning of Subsection~\ref{subsec:euler} and a remark at the end of it) we find that  one of the parameters we study  coincides with a generalization of Eulerian numbers (see \cite[sequence A060187]{s}) related to Whitney numbers of Dowling lattices (see \cite[sequences A145901, A039775]{s} and \cite{dow, ben_dm, ben_aam, clark} for definitions and  further information on these numbers). This rather unexpected and intriguing connection has not been explained and merits, perhaps, further studies. One consequence of our work is that the triangle of numbers \cite[sequence A060187]{s}, when suitably normalized, satisfies the central limit theorem. As far as we can tell this result is new (although it is an easy consequence of a general theorem of Bender \cite{b}). Limit theorems for a related sequence \cite[A145901]{s} are established in \cite{clark}. 
 This link to the Whitney numbers of Dowling lattices may have unraveled  connection to geometry alluded to in \cite{cw} as the sequences A060187, A145901, and A039775 from \cite{s} all have a very strong geometrical flavor. 

\section{definitions and notation}

We recall the following concept first introduced in \cite{cw_pnas,cw}: A {\it staircase tableau of size $n$} is a Young diagram of  shape $(n,n-1,\dots,2,1)$ whose boxes are filled according to the following rules:
\begin{itemize}
\item each box is either empty or contains one of the letters $\a$, $\b$, $\delta$, or $\g$;
\item no box on the diagonal is empty;
\item all boxes in the same row and to the left of a $\b$ or a $\delta$ are empty;
\item all boxes in the same column and above an $\a$ or a $\g$ are empty.
\end{itemize}
An example of a staircase tableau  is given in Fig.~\ref{fig:tab}(a).
 \begin{figure}[h]
\centering
\begin{tabular}{ccc}
\includegraphics[width=0.28\textwidth]{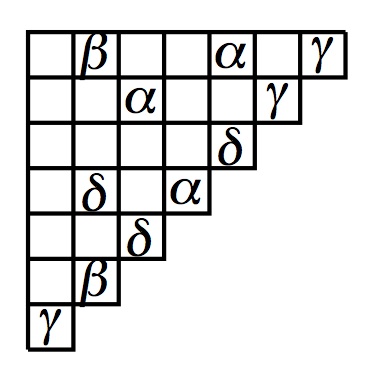}
&
\includegraphics[width=0.3\textwidth]{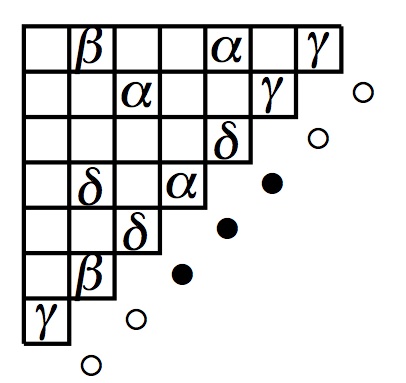}
&
\includegraphics[width=0.29\textwidth]{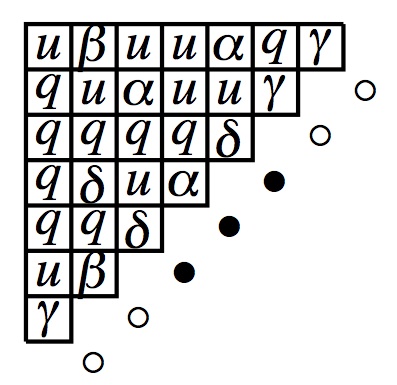}\\(a)&(b)&(c)
\end{tabular}
\caption{(a) A staircase tableau of  size is $7$. Its top row is indexed by $\beta$, the next one by $\a$. (b) Its  type  ($\circ \circ \bullet \bullet \bullet \circ \circ$). (c) The same tableau with $u$'s and $q$'s.}
\label{fig:tab} 
\end{figure}

As we mentioned, staircase (or earlier permutation) tableaux were studied in the connection with ASEP. Because of the importance of this connection we briefly recall its nature.  
The ASEP is a Markov chain on words of size $n$ on an alphabet $A = \{\circ,\bullet\}$ consisting of two letters. Each such word represents an one-dimensional lattice of length $n$ 
 with some sites occupied by  particles  (represented by $\bullet$), and others not (represented by $\circ$). A particle can only hop to the right or the left (with the probabilities $u$ and $q$, respectively), provided that the adjacent site is unoccupied, or enter or quit the lattice. Entering from the left (right) happens with the probability $\a$ (resp., $\delta$) if the first (last) site is unoccupied. Exiting to the left (right) happens with the probability $\g$ (resp., $\b$) if the first (last) site is occupied. At a given time one of the $n+1$ possible locations for a move  is selected (uniformly at random) and, if possible,  a transition described above is performed  with the given probability. We refer to \cite{Derrida, Derrida1, jumping} or \cite{cw} for more detailed description and further references. 

To describe the connection to staircase tableaux, define the {\it type}  of a staircase tableau $S$ of size $n$ to be a word of the same size on the alphabet $\{\circ,\bullet\}$ obtained by reading the
diagonal boxes from northeast (NE) to southwest  (SW) and writing $\bullet$ for each $\alpha$ or $\delta$, and $\circ$ for each $\beta$ or $\gamma$. (Thus a type of a tableau is a possible state for the ASEP.)  Fig.~\ref{fig:tab}(b) shows a tableau with its type. We also need a   \textit{weight} of a tableau $S$. To compute it, we first label  the empty boxes of $S$ with $u$'s and $q$'s as follows: first, we fill all the boxes to the left of a $\beta$ with  $u$'s, and all the boxes to the left of a $\delta$ with  $q$'s. Then, we fill the boxes above an $\alpha$ or a $\delta$ with  $u$'s, and the boxes above a $\beta$ or a $\gamma$ with  $q$'s. When the tableau is filled, its weight, $wt(S)$, is 
a monomial of degree $n(n+1)/2$ in $\alpha$, $\beta$, $\gamma$, $\delta$, $u$ and $q$, which is the product of labels of the boxes of $S$. Fig.~\ref{fig:tab}(c) shows a tableau filled with $u$'s and $q$'s. Its weight is $\a^3\b^2\delta^3\g^3u^8q^9$.

Corteel and  Williams \cite{cw,cw_pnas} have shown that the steady state probability that
the ASEP is in state $\sigma$ is 
\[\frac{Z_{\sigma}(\alpha,\beta,\gamma,\delta,q,u)}{Z_{n}(\alpha,\beta,\gamma,\delta,q,u)},\] where 
$\displaystyle Z_n(\alpha, \beta,\gamma,\delta,q,u) = \sum_{S\textrm{ of size }n } wt(S)$ and  
$\displaystyle Z_{\sigma}(\alpha,\beta,\gamma,\delta,q,u)= \sum_{S\textrm{ of type }\sigma } wt(S)$.

We denote the set of all staircase tableaux of size $n$ by $\cS_n$, $n\ge1$. It is known  that  the cardinality of $\cS_n$ is $4^nn!$. There are several proofs of this fact (c.f. \cite{cssw} for one of them  and for references to further proofs). All these proofs are based on combinatorial approaches and  we wish to mention that a probabilistic technique that we develop in this paper provides a yet another proof of that fact. We present it in Section~\ref{count} below as an illustration of how our method works.

We now define some parameters that are the objects of our study. Let $*$ be a subset of the set of symbols $\{\a,\b,\delta,\g\}$. We say that a row of a staircase tableau is {\it indexed by $*$} if  its leftmost entry is a member of $*$. For the sake of brevity we will refer to rows indexed by $*$ simply as {\it $*$ rows}. Thus, for example, the number of $\a/\g$ rows is the number of rows indexed by $\a$ or $\g$. The tableau in Figure~\ref{fig:tab}(a) has two $\a/\g$ rows, the second from the top (indexed by $\a$) and the bottom (indexed by $\g$). For a given staircase tableau $S\in\cS_n$ we denote this quantity by $r_n(S)$ and we occasionally  skip the subscript $n$ if there is no risk of confusion. As we demonstrate in Section~\ref{setup} below this parameter  plays a fundamental role in our approach. Other parameters we  consider are: the total number of entries $\beta$ or $\delta$ ($\b/\delta$ for short), the total number of entries $\a$ or $\g$ ($\a/\g$), the number of entries $\b/\delta$ on the diagonal of the tableau, and the number of entries $\a/\g$ on the diagonal. For a given tableau $S\in\cS_n$ these parameters will be denoted by $\Delta_n(S)$, $\Gamma_n(S)$, $B_n(S)$, and $A_n(S)$,  respectively.  
We gather the names and the notation we use for the parameters we consider along with their values for a tableau given in Figure~\ref{fig:tab}(a) in the following table 

\vspace{.5cm}
\begin{center}
\begin{tabular}{|c|c|c|}
\hline
parameter&notation&value  in Fig.~\ref{fig:tab}(a)\\
\hline
$\alpha/\gamma$ rows&$r_n$&2\\\hline
total \# of $\beta/\delta$&$\Delta_n$&5\\\hline
total \# of $\alpha/\gamma$&$\Gamma$&6\\\hline
\# of $\beta/\delta$ on diagonal&$B_n$&3\\\hline
\# of $\alpha/\gamma$ on diagonal&$A_n$&4\\\hline
\end{tabular}\\
\end{center}

\vspace{.5cm}
As we mentioned earlier our viewpoint is probabilistic. Thus, we  equip the set $\cS_n$ with the uniform probability measure denoted by $\P_n$. This means that for each $S\in\cS_n$ we have 
\[\P_n(S)=\frac1{4^nn!}.\]\
As is customary we refer to a tableau chosen according to that measure as  a {\it random tableau of size $n$}. We  denote the integration with respect to the measure $\P_n$ 
 by $\E_n$. Our goal is to  analyze  probabilistic properties, like expected values, variances, and exact or limiting distributions, of random variables (or {\it statistics}) $r_n$, $\Delta_n$, $\Gamma_n$, $B_n$, and $A_n$. We  follow the probability theory conventions of \cite{shiryaev}, and the reader is pointed there for any unexplained terms.

\section{Preliminaries and outline of the argument}\label{setup}

In this section we  detail main ideas beyond our approach  and also  derive basic properties of the fundamental parameter, i.e. the number of $\alpha/\gamma$ rows. 

Our method is analogous to what has been done in the case of permutation tableaux (see \cite{ch} or \cite{hj}). Let us recall at this point that permutation tableaux have been used  to give a combinatorial description of a stationary distribution for the 
PASEP. We refer to e.g. \cite{CW1, CW2,ch,hj} for the definition, connections to PASEP, further properties and details.  Just as PASEP is a particular case of ASEP, permutation tableaux of size $n$ are in bijection with a subset of staircase tableaux of that size corresponding to the case $\g=\delta=0$.

The  approach used in \cite{ch,hj} for the permutation tableaux  was to identify a fundamental parameter, trace its evolution as the size of a tableau is increased by 1, and then use successively conditioning to reduce the size of a tableau. We refer the reader to either \cite{ch} or \cite{hj} for more details, here we only recall that this fundamental parameter was the number of unrestricted rows, $U_n$, in a permutation tableau, and that its conditional distribution was given by $1+\operatorname{Bin}(U_{n-1},1/2)$ (here $\operatorname{Bin}(n,p)$ denotes a binomial random variable with parameters $n$ and $p$). This is to mean that if a size of a  permutation tableau with $U_{n-1}$ unrestricted rows was increased from $n-1$ to $n$, then the number of unrestricted rows in this extension had (the conditional) distribution $1+\operatorname{Bin}(U_{n-1},1/2)$.  As it turns out,  in the case of staircase tableaux the role of a fundamental parameter is played by the number of  $\a/\g$ rows. 

To make our approach work we need to know two things. One is the knowledge of the conditional distribution of statistics of interest in extension of a given tableau of size $n-1$. This would allow us to use the so--called tower property of the conditional expectation to reduce the size of staircase tableaux by 1. Roughly speaking, this reducing would amount to relating a function of a statistic on, say, $\mathcal S_n$ to another function of the same statistics on $\mathcal S_{n-1}$. Keeping track of how the functions are related would in principle allow for an iteration of this process to reduce the size of the tableaux to 1. There is, however, another difficulty. Namely, when passing from $\mathcal S_{n}$ to $\mathcal S_{n-1}$, not only a function of statistic changes but also a measure over which we integrate. More specifically, the image of the uniform probability measure $\P_n$  on $\mathcal S_n$ under a natural mapping associated with our reduction does not yield the uniform probability measure $\P_{n-1}$ on $\mathcal S_{n-1}$. To handle this difficulty we need a second ingredient which we refer to as the change of measure. It simply describes the relation between the two probability measures and allows to return to the uniform probability measure on $\mathcal S_{n-1}$ after the reduction (see Section~\ref{sub:ch_of_meas} for more details).     

In the reminder of this section we develop the  details for the number of $\alpha/\gamma$ rows, a key statistic in further considerations. We do  it in two separate subsections. In the first we analyze the evolution  of the number of $\alpha/\gamma$ rows as the size of the tableaux increases.  The key step is the computation of the conditional generating function of  $r_n$ as the size of the tableau is extended from $n-1$ to $n$ (Lemma~\ref{lem:gf_r} below). In the second subsection we discuss the change of measure.

\subsection{The conditional generating function for $r_n$}
Our analysis progresses in steps. We first find the number of extensions of a given tableau of size $n-1$. We then find the conditional distribution of the number of $\alpha/\gamma$ rows. This enables us to compute the conditional generating function of the number of such rows in an extension of a given tableau of size $n-1$. The formula expresses this generating function in terms of the number of $\alpha/\gamma$ rows of the tableau being extended and is of crucial importance for our approach. 

 To begin the analysis we need to  briefly recall the evolution process of staircase tableaux described by Corteel, Stanton, and Williams in \cite{csw}. 
Let $S$ be a staircase tableau of size $n-1$. To extend its size by 1, we add a new column of length $n$ at the left end and we  fill it according to the rules. Any such filling gives a tableau of size $n$ which is an extension of $S$. The number of extensions clearly depends on the number of $\alpha/\gamma$ rows of $S$ and  is as follows:
\begin{lemma}\label{numberofext} Let $S \in \cS_{n-1}$ be a tableau with $r_{n-1}=r_{n-1}(S)$ $\a/\g$ rows. Then there are $4\cdot3^{r_{n-1}}$ different staircase tableaux of size $n$ that are extensions of $S$.
\end{lemma}
\begin{proof}  We consider two cases:

\noindent{\bf 1.}  If $r_{n-1}=0$ then all $n-1$ rows of $S$ are $\b/\delta$ rows and hence the top $n-1$ boxes of the new column have to be  empty since no entries are allowed to the left of a $\b/\delta$ in the same row. 
Thus, if $r_{n-1}=0$ we obtain four different tableaux of size $n$ by putting one of the four symbols in the bottom box of the new column. 

 \begin{figure}[h]
\centering
\begin{tabular}{ccc}
\includegraphics[width=0.3\textwidth]{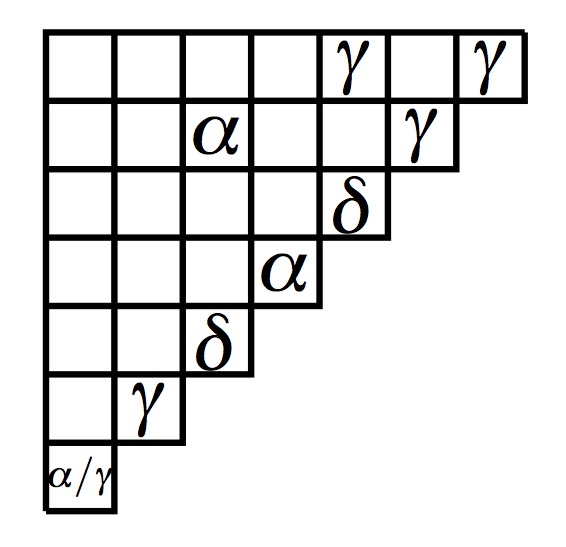}
&\hspace{2.5cm}
\includegraphics[width=0.3\textwidth]{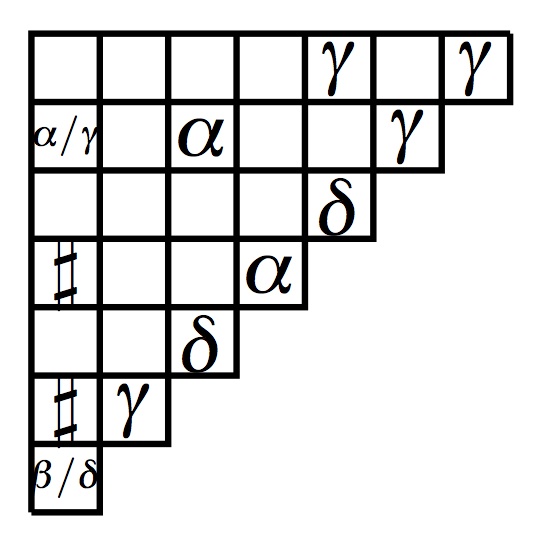}
\\(a)&\hspace{2.5cm}(b)
\end{tabular}
\caption{Extension of a tableau of size 6: (a)  $\a/\g$ is put in the bottom corner of the 7th column. All boxes above it are to be empty. (b) $\b/\delta$ is put in the bottom box and $\a/\g$ in the 3rd of the four $\a/\g$ rows in the 7th column. A box in an $\a/\g$ row above is to be empty; $\sharp$'s in two $\a/\g$ rows below it indicate that these boxes can be filled with $\b/\delta$ or left empty.}
\label{fig:tab_ext} 
\end{figure}

\noindent{\bf 2.} If $r_{n-1}\ge1$ we count the number of extensions as follows (see Figure~\ref{fig:tab_ext} for an illustration):  we can either put one of the symbols $\a/\g$ in the bottom corner (and then we are forced to leave all other boxes in that column empty), or we can put a $\b/\delta$ in the bottom box of the new column. In the latter case, we need to fill the $r_{n-1}$ boxes in the new column corresponding to the  $\a/\g$ rows of $S$. According to the rules, if we put an $\a/\g$ in any of them, then we need to leave all boxes above it empty, otherwise we have a complete freedom.    It follows from these considerations that any  tableau of size $n$ with $r_{n-1}$ $\a/\g$ rows gives $4\cdot3^{r_{n-1}}$ different staircase tableaux of size $n$.   Indeed, if $r_{n-1}=0$ then there are 4 extensions and if $r_{n-1}\ge1$ then there are 
\[2+2(2(1+3+\dots+3^{r_{n-1}-1})+3^{r_{n-1}})\]
extensions. Here, the first 2 is from putting an $\a/\g$ in the bottom box of the new column, the next 2 is from putting a $\b/\delta$ in that box, the term $2\cdot3^{i-1}$, $1\le i\le r_{n-1}$, is from putting the first  $\a/\g$ in a box of the new column corresponding to the  $i$th (counting from bottom) $\a/\g$ row of $S$  as then there are $3^{i-1}$ ways of filling the earlier $i-1$ boxes with symbols $\b$, $\delta$, or leaving them empty, and finally, the $3^{r_{n-1}}$ term comes from not putting an $\a/\g$ in any of the $r_{n-1}$ boxes and thus filling them  with $\b/\delta$'s or leaving them empty.

Summing the above gives
\[2+2\left(2\frac{3^{r_{n-1}}-1}2+3^{r_{n-1}}\right)=2+2(3^{r_{n-1}}-1+3^{r_{n-1}})=4\cdot3^{r_{n-1}},\]
as claimed.
\end{proof}
 
 As the next step we find the conditional distribution of the number of $\alpha/\gamma$ rows over the extensions of a given tableau of size $n-1$. To do that we
 phrase the preceding discussion  in a more probabilistic language using what Shiryaev \cite[Chapter~I, \S8 
]{shiryaev} refers to as  \lq\lq decompositions\rq\rq\  (which is just a special case of conditioning with respect to $\sigma$--algebra).  

Note that every tableau from $\cS_n$ is an extension of a unique tableau $S$ from $\cS_{n-1}$. Therefore,   denoting by $D_S$ the set of all tableaux from $\cS_n$ which are obtained from $S$ by the process described above  (as we just discussed, there are $4\cdot3^{r_{n-1}(S)}$ such tableaux),  we can write
\[\mathcal S_n=\bigcup_{S\in\cS_{n-1}}D_S,\]
where   $(D_{S})_{S\in\cS_{n-1}}$ are pairwise disjoint, non--empty subsets of $\mathcal S_n$.
We denote this decomposition of $\cS_n$ by $\mathcal D_{n-1}$.
We wish  to compute the conditional probabilities  $\P(\ \cdot\ |\mathcal D_{n-1})$ and the conditional expectations  $\E(\ \cdot\ |\mathcal D_{n-1})$  with respect to this decomposition. We have
\begin{lemma}\label{lem:condist} Let $S\in\cS_{n-1}$ be a  tableau with $r_{n-1}=r_{n-1}(S)$ $\a/\g$ rows and let $r$ be the number of such rows in any of its 
 extensions
to a tableau of size $n$. Then 
\begin{equation}\label{plus1}\P(r=r_{n-1}+1|D_{S})=\frac1{2\cdot3^{r_{n-1}}},\end{equation}
and, for $k=0,1,\dots, r_{n-1}$
\begin{equation}\label{minus_k}\P(r=r_{n-1}-k|D_{S})=\frac{2^{k+1}(2{r_{n-1}\choose k+1}+{r_{n-1}\choose k})}{4\cdot3^{r_{n-1}}}.
\end{equation}
\end{lemma}
\begin{proof}
Let $S$ and $r$ be as in the statement. Clearly, the possible values for $r$ are $r_{n-1}+1, r_{n-1}, r_{n-1}-1,\dots, 1, 0$ and we need to know the probability of each of these possibilities. First, $r=r_{n-1}+1$ means that we must have placed an $\alpha/\gamma$ in the bottom box of the new column since this is the only way we can increase the number of $\alpha/\gamma$ rows. Consequently, all other boxes in the new column are to remain empty. Obviously there are two such extensions which in view of Lemma~\ref{lem:condist} gives 
\[\P(r=r_{n-1}+1|D_{S})=\frac2{4\cdot3^{r_{n-1}}}=\frac1{2\cdot3^{r_{n-1}}},\]
which justifies \eqref{plus1}.

Next, for $k=0,1,\dots, r_{n-1}$ we compute $\P(r=r_{n-1}-k|D_{S})$.
Since $k$ is the number of $\b/\delta$'s that we put in the $r_{n-1}$ allowable (i.e. corresponding to $\a/\g$ rows of $S$) boxes, $r=r_{n-1}-k$ means that we put a $\b/\delta$ in the bottom box and additional $k$ $\b/\delta$'s in the $r_{n-1}$ allowable boxes above it. If we do not put an $\a/\g$ in any of those $k$ boxes,  we have $2\cdot2^k{r_{n-1}\choose k}$ possibilities (2 is for putting either a $\b$ or a $\delta$ at the bottom, and the rest accounts for putting $\b/\delta$'s in any $k$ of the allowable $r_{n-1}$ boxes). If we do put an $\a/\g$ in one of the boxes we need to pick $k+1$ of the allowable $r_{n-1}$ boxes, put an $\a/\g$ in the topmost of them and put $\b/\delta$'s in the remaining $k$ of them (and at the bottom). This gives  $2^{k+2}{r_{n-1}\choose k+1}$ possibilities. Adding up the two pieces and dividing by the total number of extensions given in Lemma~\ref{lem:condist} we obtain, 
\[\P(r=r_{n-1}-k|D_{S})=\frac{2\cdot2^k{r_{n-1}\choose k}+2^{k+2}{r_{n-1}\choose k+1}}{4\cdot3^{r_{n-1}}}=\frac{2^{k+1}(2{r_{n-1}\choose k+1}+{r_{n-1}\choose k})}{4\cdot3^{r_{n-1}}}
\]
as required.\end{proof}

Lemma~\ref{lem:condist} completely describes the conditional distribution of $r_n$ given $\mathcal D_{n-1}$  and leads, in particular,  to the following basic relation: 
\begin{lemma}\label{lem:gf_r}
For a complex number $z$ and $n\ge1$ (with the understanding that $r_0\equiv0$)
\begin{equation}\label{gf_r}\E(z^{r_n}|\mathcal D_{n-1})=\frac{z+1}2\left(\frac{z+2}3\right)^{r_{n-1}}.\end{equation}
\end{lemma}
\begin{proof}
Recall that as we extend a tableau with $r_{n-1}$ $\a/\g$ rows, we have $r_n=r_{n-1}+1$ if and only if we put an $\a/\g$ at the bottom of the new column and $r_n=r_{n-1}-k$ if and only if we put a $\b/\delta$ there and additional $k$ $\b/\delta$'s in the $r_{n-1}$ boxes above it. Therefore, letting $C_*$ denote the event that we put the symbol * at the bottom of the new column we have
\begin{eqnarray*}\E(z^{r_n}|\mathcal D_{n-1})&=&\E(z^{r_n}I_{C_{\a/\g}}|\mathcal D_{n-1})+\E(z^{r_n}I_{C_{\b/\delta}}|\mathcal D_{n-1})\\&=&
z^{r_{n-1}+1}\frac1{2\cdot3^{r_{n-1}}}+\sum_{k=0}^{r_{n-1}}z^{r_{n-1}-k}\P(r_n=r_{n-1}-k|\mathcal D_{n-1})\\&=&
\frac z2\left(\frac z3\right)^{r_{n-1}}+\frac{2}{4\cdot3^{r_{n-1}}}\sum_{k=0}^{r_{n-1}}z^{r_n-k}2^k
\left(2{r_{n-1}\choose k+1}+{r_{n-1}\choose k}\right)
\\&=& 
\frac z2\left(\frac z3\right)^{r_{n-1}}+\frac{1}2\left(\frac{2+z}3\right)^{r_{n-1}}
+\frac{1}{3^{r_{n-1}}}\sum_{k=0}^{r_{n-1}-1}z^{r_n-k}2^k{r_{n-1}\choose k+1}.
\end{eqnarray*}
The last term is
\[\frac{z}{2\cdot3^{r_{n-1}}}\sum_{k=0}^{r_{n-1}-1}2^{k+1}z^{r_n-k-1}{r_{n-1}\choose k+1}
=\frac{z}{2\cdot3^{r_{n-1}}}\left((2+z)^{r_{n-1}}-z^{r_{n-1}}\right),
\]
so that we obtain
\begin{eqnarray*}\E(z^{r_n}|\mathcal D_{n-1})&=&\frac z2\left(\frac z3\right)^{r_{n-1}}+\frac{1}2\left(\frac{2+z}3\right)^{r_{n-1}}
+
\frac z2\left(\frac{2+z}3\right)^{r_{n-1}}-\frac z2\left(\frac z3\right)^{r_{n-1}}\\&=&
\frac{1+z}2\left(\frac{2+z}3\right)^{r_{n-1}}.
\end{eqnarray*}
This proves \eqref{gf_r}. 
\end{proof}
As we already mentioned, Lemma~\ref{lem:gf_r} (or its  versions)  plays a crucial role in our approach. 

\subsection{The change of measure}\label{sub:ch_of_meas}
In this subsection we describe the  second ingredient, namely the change of measure. It it is necessitated by the fact that there are two different probability measures on $\mathcal S_{n-1}$  that naturally appear in our considerations.  The first  is, of course, the uniform measure  $\P_{n-1}$.  We discuss the second one and the relation between them below.

Consider $\mathcal S_{n-1}$, the set of all staircase tableaux of size $n-1$. The second probability measure on $\mathcal S_{n-1}$ is obtained from the uniform probability measure on $\mathcal{S}_n$ by "collapsing" all the elements of $\mathcal{S}_n$  that are extensions of the same element  $S\in\mathcal{S}_{n-1}$. More formally, consider a map $f:\mathcal S_n\to\mathcal S_{n-1}$ defined by $f(T)=S$ if and only if $T$ is an extension of $S$.  The measure of interest is the image of $\P_n$ under $f$. We  denote this measure $\P_n(S)$ (there is an apparent ambiguity of notation here, however, it disappears once we remember whether $S$ is in $\mathcal S_n$ or $\mathcal S_{n-1}$). Since both of these measures appear in the course of our argument, it is important to find the relationship between them.  But this is  straightforward: since a tableau $S\in\cS_{n-1}$ with $r_{n-1}$ $\a/\g$ rows gives $4\cdot 3^{r_{n-1}}$ tableaux in $\cS_{n}$ we have
\begin{equation}\label{rel_prob}\P_n(S)=\frac{4\cdot3^{r_{n-1}}}{|\cS_n|}=\frac{4\cdot3^{r_{n-1}}|\cS_{n-1}|}{|\cS_n|}\frac1{|S_{n-1}|}=4\cdot3^{r_{n-1}}\frac{|\cS_{n-1}|}{|\cS_n|}\P_{n-1}(S).\end{equation}
Consequently, for any random variable $X_{n-1}$ on $\cS_{n-1}$ we have
\begin{equation}\label{rel_exp}\E_nX_{n-1}=\E_{n-1}4\cdot3^{r_{n-1}}\frac{|\cS_{n-1}|}{|\cS_n|}X_{n-1}=4\frac{|\cS_{n-1}|}{|\cS_n|}\E_{n-1}3^{r_{n-1}}X_{n-1}.\end{equation}
Here we have used the same convention as above; for a random variable $X$ on $\mathcal S_{n-1}$, $\E_{n-1}X$ denotes the expectation with respect to the uniform measure on $\mathcal S_{n-1}$ while $\E_nX$ denotes the expectation with respect to the measure that is induced on $\mathcal S_{n-1}$ by the uniform measure on $\mathcal S_n$.
  
The relations \eqref{gf_r} and \eqref{rel_exp} are key and will allow us to analyze the distributions of the various statistics on $\mathcal S_n$.  Note that  \eqref{rel_prob} and \eqref{rel_exp} are true regardless of whether we know  the cardinalities of  $\mathcal S_{n-1}$ and $\mathcal S_n$ or not. As a matter of fact, one can use \eqref{rel_exp} to provide a yet another argument that $|\mathcal S_n|=4^nn!$ as we will see in the next section. 

\section{Generating function for the number of $\a/\g$ rows and some consequences}
\label{count}
In this section we demonstrate  how we intend to apply \eqref{gf_r} and \eqref{rel_exp} to derive recurrences for generating functions that upon solving yield information on the corresponding statistics. We focus here on the number of   $\a/\g$ rows which, on one hand is the easiest to analyse, and on the other is central in the analysis of other statistics.
\begin{proposition}\label{prop:gfr}
For every complex number $z$ we have
\begin{equation}\label{eq:gfr}
E_nz^{r_n}=\frac{2^n}{|\mathcal S_n|}\prod_{k=1}^n(z+2k-1).
\end{equation}
 \end{proposition}
 \begin{proof}
By the basic properties of the conditional expectation (see e.g. \cite[Formula~(16), p.~79]{shiryaev}) the expectation on the right--hand side is equal to 
 $\E_n\E(z^{r_n}|\mathcal D_{n-1})$.
Using first \eqref{lem:gf_r} and  then \eqref{rel_exp} we further get
\begin{eqnarray*}\E_nz^{r_n}&=&
\frac{z+1}2\E_n\left(\frac{z+2}3\right)^{r_{n-1}}=\frac{(z+1)}2\frac{4|\mathcal S_{n-1}|}{|\mathcal S_n|}\E_{n-1}3^{r_{n-1}}\left(\frac{z+2}3\right)^{r_{n-1}}\\&=&2(z+1)\frac{|\mathcal S_{n-1}|}{|\mathcal S_n|}E_{n-1}(z+2)^{r_{n-1}}.
\end{eqnarray*}
Applying the same procedure with $z$ replaced by $z+2$ and $n$ by $n-1$ we obtain
\[E_nz^{r_n}=2^2(z+1)(z+3)\frac{|\mathcal S_{n-2}|}{|\mathcal S_n|}E_{n-2}(z+4)^{r_{n-2}}.
\]
Further iteration yields
\[\E_nz^{r_n}=2^{n-1}(z+1)(z+3)\dots(z+2n-3)\frac{|\mathcal S_{1}|}{|\mathcal S_n|}\E_1(z+2(n-1))^{r_1}.\]
Since $|\mathcal S_1|=4$ and  $\E_1(z+2(n-1))^{r_1}=\frac12(z+2(n-1))+\frac12$ the above can be written as
\[\E_nz^{r_n}=\frac{2^n}{|\mathcal S_n|}(z+1)(z+3)\dots(z+2n-1),\]
which is precisely \eqref{eq:gfr}. 
\end{proof}
Proposition~4 has a number of consequences. First, by putting $z=1$ in \eqref{eq:gfr}  we obtain an independent confirmation of the count of staircase tableau of a given size.  
 \begin{corollary}\label{cor:count}
 Let  $\mathcal S_n$ be the set of all staircase tableaux of size $n\ge1$. Then 
\[|\mathcal S_n|=4^nn!
\]
 \end{corollary}
Note that once this corollary is known \eqref{rel_prob} and \eqref{rel_exp} simplify to
\begin{equation}\label{rel_simp}
\P_n(S)=\frac{3^{r_{n-1}}}n\P_{n-1}(S)\quad\mbox{and}\quad \E_nX_{n-1}=\frac1n\E_{n-1}3^{r_{n-1}}X_{n-1},
\end{equation}
respectively, and  this is the form we will be using from now on.

Next, by combining this corollary with \eqref{eq:gfr} we obtain
\begin{corollary}\label{cor:gfr}
The probability generating function of the number of $\alpha/\gamma$ rows in a random staircase tableau of size $n$ is given by
\[\E_nz^{r_n}=\frac{2^n}{4^nn!}\prod_{k=1}^n(z+2k-1)=\prod_{k=1}^n\frac{z+2k-1}{2k}.\]
\end{corollary}
The last corollary gives, in turn, a complete information on the distribution of $r_n$.
\begin{corollary}\label{cor:arows_dist}
For every $n\ge1$ we have
\begin{equation}\label{arows_dist}r_n\stackrel d=\sum_{k=1}^nJ_k,\end{equation}
where $J_k$'s are independent and $J_k$ is 
a random variable  which is 1 with probability $1/(2k)$ and 0 with the remaining probability. In particular,
\begin{equation}\label{exp_var_arows}\E_nr_n=\sum_{k=1}^n\frac1{2k}=
\frac{H_n}2
,\quad
\operatorname{var}(r_n)=\sum_{k=1}^n\frac1{2k}\left(1-\frac1{2k}\right)=
\frac{H_n}2-\frac{H_n^{(2)}}4,
\end{equation}
where $H_n=\sum_{k=1}^n\frac1k$ and $H_n^{(2)}=\sum_{k=1}^n\frac1{k^2}$ are harmonic numbers of the first and second order, respectively. Furthermore,  
\begin{equation}\label{clt_arows}\frac{r_n-\frac{\ln n}2}{\sqrt{\frac{\ln n}2}}\stackrel d\longrightarrow N(0,1).
\end{equation}
\end{corollary}
\begin{proof} 
Note that a factor
\[\frac{z+2k-1}{2k}=\frac z{2k}+1(1-\frac1{2k}),\]
given in the previous corollary
is the probability generating function of a random variable $J_k$ which is 1 with probability $1/(2k)$ and 0 with the remaining probability. Since the product of the probability generating functions corresponds to adding independent random variables, we obtain \eqref{arows_dist} and 
 thus also  \eqref{exp_var_arows}. 
Finally,   since $J_k$ are uniformly  bounded and variances of their partial sums go to infinity, the Lindeberg condition for the central limit theorem (see e.g. \cite[Chapter~III~\S4]{shiryaev}) holds trivially. Since as $n\to\infty$,  $\E_n r_n\sim\operatorname{var}(r_n)\sim\frac{\ln n}2$, \eqref{clt_arows} holds.
\end{proof}

\section{Main results}\label{sec:mainres} 
Our technique allows us to obtain further results concerning the distributions (sometimes exact, sometimes only asymptotic)  of the statistics discussed above. We gather our results in the following two statements, concerning the total number of entries and the number of entries on the diagonal, respectively. Recall that $\Delta_n$ and $\Gamma_n$ denote the total number of $\beta/\delta$ and $\alpha/\gamma$ in the tableau of size $n$, respectively. 

\begin{theorem}\label{thm:ent}   Consider the set $\cS_n$ with the uniform probability measure $\P_n$. Then:
\begin{itemize}
\item[(i)]
For every $n\ge1$ we have
\begin{equation}\label{exact_D}\D_n\stackrel d=\sum_{k=1}^nI_k,\end{equation}
where $(I_k)$ are  independent and $\P(I_k=1)=1-\frac1{2k}$, $\P(I_k=0)=\frac1{2k}$ . In particular, 
\begin{equation}\label{exp_var_D}E_n\D_n=n-\frac{H_n}2,\quad \operatorname{var}(\D_n)=\frac{H_n}2-\frac{H_n^{(2)}}4,
\end{equation}
and, as $n\to\infty$,  
\begin{equation}\label{dist_D}\frac{\D_n-n+\frac12\ln n}{\sqrt{\frac12\ln n}}\stackrel d\longrightarrow N(0,1).
\end{equation}
\item[(ii)]For every $n\ge1$
\[\G_n\stackrel d=\D_n.\]
 In particular, \eqref{exact_D}, \eqref{exp_var_D}, and \eqref{dist_D} hold with $\D_n$ replaced by $\G_n$.
 \end{itemize}
 \end{theorem}
 Our second result concerns the entries on the diagonal. Recall that $A_n$ (resp. $B_n$) denote the number of $\alpha/\gamma$ (resp. $\beta/\delta$) on the diagonal of a tableau of size $n$.  For these parameters we get
 \begin{theorem}\label{thm:diag}  The expected value and the variance of the number $A_n$ of $\a/\g$ on the diagonal of a random staircase tableau of size $n$ are, respectively,
\begin{equation}\label{exp_var_diag}\E_nA_n=\frac n2\quad\mbox{and}\quad\operatorname{var}(A_n)=\frac{n+1}{12}.\end{equation}
Moreover, 
\begin{equation}\label{clt_diag}\frac{A_n-n/2}{\sqrt{n/12}}\stackrel d\longrightarrow N(0,1).\end{equation}
 Furthermore, for every $n\ge1$ we have
\begin{equation}\label{eq_diag}B_n\stackrel d= A_n.\end{equation} 
In particular, \eqref{exp_var_diag}  and \eqref{clt_diag}  
hold for $B_n$ in place of $A_n$.
\end{theorem}

\noindent{\bf Remark:} While it seems intuitively clear that the expected number of letters $\alpha/\gamma$ on the diagonal is $n/2$ as \eqref{exp_var_diag} asserts, the expression for the variance is much less intuitive. It implies, in particular, that the entries $\alpha/\gamma$ and $\beta/\delta$ along the diagonal are not chosen independently from one another with equal probabilities as one might have hoped (if that were the case the variance would be $n/4$).

\section{proofs}\label{proofs}
We begin by observing that, as is obvious from the definition, staircase tableaux are symmetric under taking the transpose and  exchanging the $\alpha$'s with the $\beta$'s and  the $\gamma$'s with the $\delta$'s. 
 It therefore is immediate that parts (ii) of both theorems follow from the respective parts (i). Furthermore, the proof of Theorem~\ref{thm:ent}~(i) may be completed by using  Corollary~\ref{cor:arows_dist} and the relation 
\begin{equation}\label{eq:rplusdel}r_n(S)+\Delta_n(S)=n\end{equation} since then
\[\Delta_n=n-r_n=\sum_{k=1}^n(1-J_k)\stackrel d=\sum_{k=1}^nI_k.\]
But \eqref{eq:rplusdel} is clear once we 
 notice that  for any row of a  staircase tableau exactly one of the following statements is true
\begin{itemize}
\item it contains a $\b/\delta$
\item it is indexed by $\a/\g$.
\end{itemize}
This  proves Theorem~\ref{thm:ent} and we now turn our attention to a considerably more involved proof of Theorem~\ref{thm:diag}~(i). 

\subsection{Proof of Theorem~\ref{thm:diag}~(i)}
The idea is the same as for Corollary~\ref{cor:arows_dist}  except that this time we will actually need the joint probability generating function of $A_n$ and $r_n$.  The final expression will turn out  to be substantially more complicated than what we encountered earlier and thus harder to analyze.  Nonetheless, the situation is quite analogous to the case of the number of rows in permutation tableaux (see \cite[Section~4]{hj}).   For the reader's convenience we break up our proof into several steps each of them discussed in a separate section below. We now briefly outline the major steps in the proof indicating the section in which they are treated. We begin in the forthcoming section with the derivation of the probability generating function.   Its coefficients satisfy certain recurrences. This, in particular, enables us  to derive the exact formulas for the expected value and the variance of $A_n$  (see Subsection~\ref{subsec:exp_var} below). Furthermore, the nature of these recurrences suggests that the coefficients are related to the classical Eulerian number associated with the number of rises in random permutations. In fact, our coefficients exactly match the numbers often called the \lq\lq Eulerian numbers of type B\rq\rq. We establish and discuss further this connection  in Subsection~\ref{subsec:euler}. We think it is of independent interest and perhaps worthy of further exploration. 

Once this connection is made it is then expected that the the coefficients follow the normal law (just as the classical Eulerian numbers do).  As a matter of fact, one of the proofs (although not the first) of the asymptotic normality of the classical Eulerian
numbers is via a fairly general device due to Bender (\cite{b}) and nowadays referred to as Bender's theorem. This is, indeed, the approach we take and in the subsection~\ref{subsec:bender} we verify the conditions of Bender's theorem to conclude our proof. 

\subsubsection{Bivariate generating function}

 We give the formula for the joint generating function of $r_n$ and $A_n$ in the following statement.
\begin{proposition}\label{prop:bivgf}   Let  $z,t$ be complex numbers. Then we have
\begin{equation}\label{gf_D_n}\E_nt^{A_n}z^{r_n}=\frac1{2^nn!}\sum_{k=0}^n(t-1)^{n-k}c_{n,k}(z),
\end{equation}
where the coefficients $\{c_{m,\ell}(z):\ 0\le\ell\le m\le n\}$ satisfy the following recurrence
\[c_{m+1,\ell}(z)=(z+2\ell)c_{m,\ell}(z)+(z+2(\ell-1)+1)c_{m,\ell-1}(z),\quad 1\le\ell\le m,\]
with the following boundary conditions:
\[c_{0,0}(z)=1,\quad c_{m+1,m+1}(z)=(z+2m+1)c_{m,m}(z),\quad c_{m+1,0}(z)=zc_{m,0}(z).\]
\end{proposition}
\begin{proof}
If $n=0$ and we set $c_{0,0}(z)=1$ then both sides of \eqref{gf_D_n} are 1. Otherwise,  let $I_j$ indicate the event that we put an $\a/\g$ in the $j$th  box on the diagonal  (counting from NE to SW). Then $A_n=\sum_{j=1}^nI_j$ and we have
\[\E_nt^{A_n}z^{r_n}=\E_n\E (t^{A_{n-1}+I_n}z^{r_n}|\mathcal D_{n-1})=\E_nt^{A_{n-1}}\E (t^{I_n}z^{r_n}|\mathcal D_{n-1}),\]
where we have used the basic properties of the conditional expectations (see \cite[Formulas~(16), p.~79 and (17), p.~80]{shiryaev}).
Now,  $I_n=1$ means that we put an $\a/\g$ in the SW corner. In that case we have $r_n=r_{n-1}+1$ and since this happens with probability $1/(2\cdot3^{r_{n-1}})$ we get
\[\E(t^{I_n}z^{r_n}|\mathcal D_{n-1})=tz^{r_{n-1}+1}\frac1{2\cdot3^{r_{n-1}}}+\E(t^{I_n}z^{r_n}I_{I_n=0}|\mathcal D_{n-1}).
\]
The second term is equal to
\begin{eqnarray*}\E(z^{r_n}I_{I_n=0}|\mathcal D_{n-1})&=&\E(z^{r_n}|\mathcal D_{n-1})-\E(z^{r_n}I_{I_n=1}|\mathcal D_{n-1})\\&=&\frac{z+1}2\left(\frac{z+2}3\right)^{r_{n-1}}-\frac{z^{r_{n-1}+1}}{2\cdot3^{r_{n-1}}}.
\end{eqnarray*}
Combining we obtain
\[\E(t^{A_n}z^{r_n}|\mathcal D_{n-1})=\frac{z(t-1)}2\left(\frac z3\right)^{r_{n-1}}+\frac{z+1}2\left(\frac{z+2}3\right)^{r_{n-1}}.
\]
Using the second part of \eqref{rel_simp}  leads to the basic recurrence
\begin{eqnarray*}\E_nt^{A_n}z^{r_n}&=&\frac{z(t-1)}2\E_nt^{A_{n-1}}\left(\frac z3\right)^{r_{n-1}}+\frac{z+1}2\E_nt^{A_{n-1}}\left(\frac{z+2}3\right)^{r_{n-1}}\\&=&\frac1{2n}\left\{
z(t-1)\E_{n-1}t^{A_{n-1}}z^{r_{n-1}}+(z+1)\E_{n-1}t^{A_{n-1}}(z+2)^{r_{n-1}}\right\}.
\end{eqnarray*}
Upon further iteration of this relation we obtain for any $0\le m<n$:
\[\E_nt^{A_n}z^{r_n}=\frac1{2^mn(n-1)\dots(n-m+1)}\sum_{\ell=0}^m(t-1)^{m-\ell}c_{m,\ell}\E_{n-m}t^{A_{n-m}}(z+2\ell)^{r_{n-m}},
\]
for some  coefficients $c_{m,\ell}=c_{m,\ell}(z)$, $0\le\ell\le m$.
To see that they satisfy the stated recurrence, we apply the basic recurrence (with $n-m$ instead of $n$ and $z+2\ell$ instead of $z$) to the expectation on the right--hand side above. We get that it is equal to 
\begin{eqnarray*}&&
\frac1{2(n-m)}\Big\{(z+2\ell)(t-1)\E_{n-m-1}t^{A_{n-m-1}}(z+2\ell)^{r_{n-m-1}}\\&&\qquad\qquad+(z+2\ell+1)\E_{n-m-1}t^{A_{n-m-1}}(z+2(\ell+1))^{r_{n-m-1}}\Big\}.\end{eqnarray*}
Substituting this into the above formula for $\E_nt^{A_n}z^{r_n}$ and multiplying both sides by $2^{m+1}n(n-1)\cdot\dots\cdot(n-m)$ (to avoid writing a denominator  on the right--hand side) we get
\begin{eqnarray*}&&\sum_{\ell=0}^m(t-1)^{m-\ell}c_{m,\ell}\Big\{
(z+2\ell)(t-1)\E_{n-m-1}t^{A_{n-m-1}}(z+2\ell)^{r_{n-m-1}}\\&&
\quad\quad+(z+2\ell+1)\E_{n-m-1}t^{A_{n-m-1}}(z+2(\ell+1))^{r_{n-m-1}}
\Big\}.
\end{eqnarray*}
Splitting this in two sums, isolating the $\ell=0$ term in the first, and the $\ell=m$ in the second,  and then shifting the index in the second sum, we further obtain 
\begin{eqnarray*}&&(t-1)^mc_{m,0}z(t-1)\E_{n-(m+1)}t^{A_{n-(m+1)}}z^{r_{n-(m+1)}}\\&&\quad+\sum_{\ell=1}^m(t-1)^{m-\ell}c_{m,\ell}
(z+2\ell)(t-1)\E_{n-m-1}t^{A_{n-m-1}}(z+2\ell)^{r_{n-m-1}}\\&&
\quad+\sum_{\ell=0}^{m-1}(t-1)^{m-\ell}c_{m,\ell}(z+2\ell+1)\E_{n-(m+1)}t^{A_{n-(m+1)}}(z+2(\ell+1))^{r_{n-(m+1)}}\\&&
\quad+ c_{m,m}(z+2m+1)\E_{n-(m+1)}t^{A_{n-(m+1)}}(z+2(m+1))^{r_{n-(m+1)}}\\
&&=
(t-1)^{m+1}c_{m,0}z\E_{n-(m+1)}t^{A_{n-(m+1)}}z^{r_{n-(m+1)}}\\&&\quad+\sum_{\ell=1}^m(t-1)^{m+1-\ell}c_{m,\ell}
(z+2\ell)\E_{n-(m+1)}t^{A_{n-(m+1)}}(z+2\ell)^{r_{n-(m+1)}}\\&&
\quad+\sum_{\ell=1}^{m}(t-1)^{m-(\ell-1)}c_{m,\ell-1}(z+2(\ell-1)+1))\E_{n-(m+1)}t^{A_{n-(m+1)}}(z+2\ell)^{r_{n-(m+1)}}
\\&&
\quad+ c_{m,m}(z+2m+1)\E_{n-(m+1)}t^{A_{n-(m+1)}}(z+2(m+1))^{r_{n-(m+1)}}.
\end{eqnarray*}
So, if we write it as 
\[\sum_{\ell=0}^{m+1}(t-1)^{m+1-j}c_{m+1,\ell}\E_{n-(m+1)}t^{A_{n-(m+1)}}(z+2\ell)^{r_{n-(m+1)}},\]
as postulated above, we see that
\[ c_{m+1,0}=zc_{m,0},\quad c_{m+1,m+1}=(z+2m+1)c_{m,m}\] and, upon combining the two middle sums, that 
\[c_{m+1,\ell}=(z+2\ell)c_{m,\ell}+(z+2(\ell-1)+1)c_{m,\ell-1},\quad\mbox{for}\quad1\le \ell\le m.
\]
Taking $m=n-1$ we get 
\[\E_nt^{A_n}z^{r_n}=\frac1{2^{n-1}n!}\sum_{\ell=0}^{n-1}(t-1)^{n-1-\ell}c_{n-1,\ell}\E_1t^{A_1}(z+2\ell)^{r_1},
\]
and since 
\[\E_1t^{A_1}(z+2\ell)^{r_1}=\frac12t(z+2\ell)+\frac12=\frac12(t-1)(z+2\ell)+\frac12(z+2\ell+1),
\]
we can write
\[\E_nt^{A_n}z^{r_n}=\frac1{2^nn!}\sum_{k=0}^n(t-1)^{n-k}c_{n,k},
\]
where the coefficients $\{c_{m,\ell}:\ 0\le\ell\le m\le n\}$ satisfy the stated recurrence and the boundary conditions. The proof is complete.
\end{proof}

\subsubsection{Expected value and the variance}\label{subsec:exp_var}
A number of properties of $(A_n)$ can be deduced from Proposition~\ref{prop:bivgf}. We illustrate this by obtaining the exact expression for the expected value and for the variance of $A_n$. To help facilitate that we   put the coefficients $(c_{n,k})$ in a Pascal type triangle.
$$
\begin{array}{ccccc}
& &c_{0,0} &&\\
&&\swarrow\hspace{.3cm}\searrow&&\\
&&&&\\
&c_{1,0}&&c_{1,1}&\\
&&&&\\
&\swarrow\hspace{.4cm}\searrow&&\swarrow\hspace{.3cm}\searrow&\\
 &\dots&\dots&\dots&\\
 &\dots&\dots&\dots&\\
\swarrow &\searrow\hspace{.2cm}\swarrow&\dots&\searrow\hspace{.2cm}\swarrow&\searrow
\\
c_{n,0} &c_{n,1}&\dots&c_{n,n-1}&c_{n,n}
\end{array}
$$
The SW move from a coefficient $c_{m,\ell}$ has weight 
$z+2\ell$ and a SE move  has  weight $z+2\ell+1$. The  value of a given coefficient is obtained by summing over all possible paths leading to it from the root $c_{0,0}$ the products of weighs  corresponding to the moves along the path. For example, there is only one path leading to $c_{n,n}$ (all moves are SE)  and hence
$$c_{n,n}=\prod_{j=0}^{n-1}(z+2j+1).$$
Likewise, paths leading to $c_{n,n-1}$ have exactly one SW move; thus there are $n$ of them and if the sole  SW move is from $c_{k,k}$, $0\le k\le n-1$, then the weight of that path is
\[\left(\prod_{j=0}^{k-1}(z+2j+1)\right)(z+2k)\left(\prod_{j=k}^{n-2}(z+2j+1)\right).\]  
Consequently,
\[c_{n,n-1}(z)=\sum_{k=0}^{n-1}
\left(\prod_{j=0}^{k-1}(z+2j+1)\right)(z+2k)\left(\prod_{j=k}^{n-2}(z+2j+1)\right).\]  
The significance of this is that
\begin{eqnarray*}\E_nA_n&=&\frac\partial{\partial t}\E_nt^{A_n}z^{r_n}_{\big|t=1,z=1}=\frac1{2^nn!}c_{n,n-1}(1)
\\&
=&\frac1{2^nn!}
\sum_{k=0}^{n-1}
\left(\prod_{j=0}^{k-1}2(j+1)\right)(2k+1)\left(\prod_{j=k}^{n-2}2(j+1)\right).
\\&
=&\frac1{2^nn!}
2^{n-1}(n-1)!
\sum_{k=0}^{n-1}(2k+1)=\frac n2.
\end{eqnarray*}
This proves the first part of \eqref{exp_var_diag}.
Similarly, to compute the variance we use $\operatorname{var}(A_n)=\E A_n(A_n-1)+\E A_n-(\E A_n)^2$ and the fact that  $\E A_n(A_n-1)=\frac{\partial^2}{\partial t^2}(\E t^{A_n})_{|t=1}=\frac2{2^nn!}c_{n,n-2}(1)$.
Paths leading to $c_{n,n-2}(1)$ have exactly two SW moves, the first could be from any $c_{k,k}$, $0\le k\le n-2$ and the second from $c_{\ell,\ell-1}$ for some $k<\ell\le n-1$. These two moves  have weights $2k+1$ and $2\ell-1$ respectively, and the remaining SE moves  have weights
\[2\cdot1,2\cdot2,\dots, 2k,2(k+1),2(k+2)\dots,2(\ell-1),2\ell,\dots,2(n-1).
\]
Therefore,
\begin{eqnarray*}c_{n,n-2}(1)&=&\sum_{k=0}^{n-2}\sum_{\ell=k+1}^{n-1}(2k+1)(2\ell-1)2^{n-2}(n-2)!\\&=&2^{n-2}(n-2)!\sum_{\ell=1}^{n-1}(2\ell-1)\sum_{k=0}^{\ell-1}(2k+1)
\\&=&
2^{n-2}(n-2)!\sum_{\ell=1}^{n-1}(2\ell-1)\ell^2\\&=&2^{n-2}(n-2)!\frac16n(n-1)(3n^2-5n+1)
\\&=&2^{n-2}n!(\frac12n^2-\frac56n+\frac16).
\end{eqnarray*}
Hence
\[\operatorname{var}(A_n)=2\frac{2^{n-2}n!}{2^nn!}(\frac12n^2-\frac56n+\frac16)+\frac n2-\frac{n^2}4
=\frac{n^2}4-\frac5{12}n+\frac1{12}+\frac{6n}{12}-\frac{n^2}4=\frac{n+1}{12}.
\]
Thus we have proved the second part of \eqref{exp_var_diag} as well.

\subsubsection{Connections to generalized Eulerian numbers}\label{subsec:euler}
Before establishing the asymptotic normality of $(A_n)$ we take a closer look at the doubly indexed sequence $\{c_{n,k}\}$ since it has intriguing connections that may be of interest in their own right. It is featured as  entry A145901 in \cite{s} and is closely related to another sequence from \cite{s}, namely A039755. More precisely, $c_{n,k}=2^kk!W_2(n,k)$ where ${W_m(n,k)}$ are the Whitney numbers of the second kind  satisfying the recurrence 
\[W_m(n,k)=(mk+1)W_m(n-1,k)+W_m(n-1,k-1).
\]
The numbers $(W_m(n,k))$ were introduced in \cite{dow} and their properties were studied in \cite{ben_dm,ben_aam,ben_log_conc}. 
 Since we are dealing exclusively with the case $m=2$ we  drop the subscript and we  write $W(n,k)$ for $W_2(n,k)$.
Thus, the generating function of $A_n$ may be written as
\begin{equation}\label{psit-1}\psi_n(t)=\E_nt^{A_n}=\frac1{2^nn!}\sum_{k=0}^n2^kk!W(n,k)(t-1)^{n-k}.
\end{equation}
It is perhaps of interest to mention that the numbers $(2^kk!W(n,k))$ themselves satisfy the central (and local) limit theorem as was shown in \cite{clark}. However, this is not exactly what we want since the generating function above is in powers of $t-1$ rather than $t$. In  terms of powers of $t$ $\psi_n(t)$ has the following form.
\begin{proposition}\label{prop:gfr} The probability generating function of the number of $\alpha/\gamma$ entries on the diagonal of a random staircase tableau of size $n$ has the form
\begin{equation}\label{gf_An}\psi_n(t)=
\frac1{2^nn!}\sum_{m=0}^nV(n,m)t^m,\end{equation}
where the numbers $\{V(n,m),\  0\le m\le n\}$ satisfy the boundary condition $V(n,0)=1$, the symmetry relation $V(n,m)=V(n,n-m)$, and the recurrence 
\begin{equation}\label{v-recur}V(n,m)=(2m+1)V(n-1,m)+(2(n-m)+1)V(n-1,m-1).\end{equation}
The explicit expression for $V(n,m)$ is given by
\begin{equation}\label{vnm}V(n,m)=\sum_{k=0}^{n-m}2^kk!W(n,k){n-k\choose m}(-1)^{n-k-m},\quad 0\le m\le n.\end{equation}
\end{proposition}
\begin{proof}
We start out with \eqref{psit-1}. To rewrite it in  powers of $t$, note that the $m$th derivative of $\psi_n$ at $t=0$ is
\begin{eqnarray*}\psi_n^{(m)}(0)&=&\frac1{2^nn!}\sum_{k=0}^{n-m}2^kk!W(n,k)(n-k)\cdot\dots\cdot(n-k-(m-1))(-1)^{n-k-m}\\
&=&\frac1{2^nn!}\sum_{k=0}^{n-m}2^kk!W(n,k)\frac{(n-k)!}{(n-k-m)!}(-1)^{n-k-m}.\end{eqnarray*}
Therefore, 
\[\psi_n(t)=\sum_{m=0}^n\frac{\psi_n^{(m)}(0)}{m!}t^m=
\frac1{2^nn!}\sum_{m=0}^nV(n,m)t^m\]
which shows that 
\eqref{gf_An} holds with $V(n,m)$ given by \eqref{vnm}.

It remains to verify the claimed properties  of the numbers $V(n,m)$, $0\le m\le n$. 
Since the symmetry condition follows by induction from the recurrence 
it suffices to verify the recurrence and the boundary condition. For the boundary condition, we see immediately that for $n\ge0$ $V(n,n)=W(n,0)=1$, so that once we verify the recurrence (and thus also the symmetry) we will have that $V(n,0)=1$ for all $n\ge0$.
To verify \eqref{v-recur}, we use the basic recurrence for $W(n,k)$'s to write the left--hand side of \eqref{v-recur} as
\begin{eqnarray*}
V(n,m)&=&\sum_{k=0}^{n-m}2^kk!\Big((2k+1)W(n-1,k)+W(n-1,k-1)\Big){n-k\choose m}(-1)^{n-m-k}\\&=&
\sum_{k=0}^{n-m}2^kk!(2k+1)W(n-1,k){n-k\choose m}(-1)^{n-m-k}\\ && \quad+
\sum_{k=1}^{n-m}2^kk!W(n-1,k-1){n-k\choose m}(-1)^{n-m-k}\\ &=& 
\sum_{k=0}^{n-m}2^kk!(2k+1)W(n-1,k){n-k\choose m}(-1)^{n-m-k}\\ && \quad+
\sum_{k=0}^{n-m-1}2^{k+1}(k+1)!W(n-1,k){n-k-1\choose m}(-1)^{n-m-k-1}\\ &=&
2^{n-m}(n-m)!(2(n-m)+1)W(n-1,n-m)\\&&\quad+
\sum_{k=0}^{n-m-1}2^kk!(2k+1)W(n-1,k){n-k\choose m}(-1)^{n-m-k}
\\&&\quad-
\sum_{k=0}^{n-m-1}2^kk!(2(k+1))W(n-1,k){n-k-1\choose m}(-1)^{n-m-k}.
\end{eqnarray*}
On the other hand, the right--hand side of \eqref{v-recur} is
\begin{eqnarray*}&&(2m+1)\sum_{k=0}^{n-m-1}2^kk!W(n-1,k){n-1-k\choose m}(-1)^{n-m-k-1}\\&&\quad
+
(2(n-m)+1)\sum_{k=0}^{n-m}2^kk!W(n-1,k){n-k-1\choose m-1}(-1)^{n-m-k}
\\&=&
(2m+1)\sum_{k=0}^{n-m-1}2^kk!W(n-1,k){n-k-1\choose m}(-1)^{n-m-k-1}
\\&&\quad+
(2(n-m)+1)\sum_{k=0}^{n-m-1}2^kk!W(n-1,k){n-k-1\choose m-1}(-1)^{n-m-k}
\\&&\quad+
(2(n-m)+1)2^{n-m}(n-m)!W(n-1,n-m).
\end{eqnarray*}
So, we see that  the coefficients in front of $W(n-1,n-m)$ in both expressions are the same, and to complete the verification of \eqref{v-recur} we need to see that the coefficients are the same for the remaining values of $k$, $0\le k\le n-m-1$. Cancelling the common factors, we need to see that
\begin{eqnarray*}&&(2k+1){n-k\choose m}-2(k+1){n-k-1\choose m}\\&&\quad=
(2(n-m)+1){n-k-1\choose m-1}-(2m+1){n-k-1\choose m}.
\end{eqnarray*}
But that is straightforward: using ${n-k\choose m}={n-k-1\choose m}+{n-k-1\choose m-1}$ and grouping the terms this boils down to verifying that
\[
{n-k-1\choose m}(2k+1-2(k+1)+2m+1)=
{n-k-1\choose m-1}(2(n-m)+1-2k-1),
\]
or, equivalently, that
\[
m{n-k-1\choose m}
=(n-m-k){n-k-1\choose m-1},
\]
which follows immediately from  the defining property of the binomial coefficients.
\end{proof}

\noindent{\bf Remark:} Triangle of numbers $V(n,m)$, $0\le m\le n$ is featured in the Online Encyclopedia of Integer Sequences \cite{s} as a sequence A060187  (with the shift in indexing: $V(n,m)=T(n+1,m+1)$) and is called ''Eulerian numbers of type B''. This sequence can be traced back in the literature to MacMahon's paper \cite{macmahon} and it was subsequently studied 
 in more detail in \cite[Sec. 3.2]{pyr} as a sequence $B_{n,k}(1)$. In particular, it appears that  the expression for a bivariate generating function of $(V(n,k))$  was derived for the first time in  \cite{pyr}. We use this expression in the next section to derive the asymptotic normality of $(A_n)$. 

\subsubsection{Conclusion of the proof by Bender's theorem}\label{subsec:bender}
Using the properties of the numbers $V(n,m)$  given in Proposition~\ref{prop:gfr} (and the identification with the sequence A060187 from \cite{s}) we can complete the proof of Theorem~\ref{thm:diag} by establishing  \eqref{clt_diag}.
To do that we will rely on a general   theorem due to  Bender \cite[Theorem~1]{b}. (Bender result is also described in \cite{fs}; see Section~IX.6 in general, and Theorem~IX.9, Example~IX.12, and Proposition~IX.9 in particular). Recall from \cite{s} or  \cite[Formula~(3.23)]{pyr} (and see Section~3 of \cite{pyr}  for a proof) that the bivariate generating function of the numbers $(V(n,k))$, called in \cite{pyr} $(B_{n,k}(1))$, is
\[\sum_{n\ge0}\sum_{k=0}^nV(n,k)\frac{w^kz^n}{n!}=\frac{(1-w)e^{(1-w)z}}{1-we^{2(1-w)z}}.\]
Therefore, it follows from \eqref{gf_An} that the bivariate probability generating function of the sequence $(A_n)$ is 
\[f(z,w):=\sum_{n\ge0}\psi_n(z)z^n
=\sum_{n\ge0}\sum_{k=0}^nV(n,k)\frac{w^k}{n!}\left(\frac z2\right)^n=\frac{(1-w)e^{(1-w)z/2}}{1-we^{(1-w)z}},\]
where we  define $f(z,1)=1/(1-z)$. We now closely follow the way Bender applied his result. First, 
\[f(z,e^s)=\frac{(1-e^s)e^{(1-e^s)z/2}}{1-e^se^{(1-e^s)z}}\]
has a simple pole at $z=r(s)=s/(e^s-1)$. Furthermore,
\[e^{(1-e^s)z}=e^{(1-e^s)(z-r(s))}e^{(1-e^s)r(s)}=
e^{(1-e^s)(z-r(s))}e^{-s},
\]
so that
\[f(z,e^s)=\frac{(1-e^s)e^{-s/2}e^{(1-e^s)(z-r(s))/2}}{1-e^{(1-e^s)(z-r(s))}}
=\frac{(1-e^s)e^{-s/2}}{e^{(e^s-1)(z-r(s))/2}-e^{-(e^s-1)(z-r(s))/2}}.
\]
 Since for bounded $u$
 \[\frac1{e^u-e^{-u}}=\frac1{2u}+O(1),\]
we get
\[f(z,e^s)=\frac{(1-e^s)e^{-s/2}}{(e^s-1)(z-r(s))}+O(1),\]
with the constant in $O(1)$ bounded when both $s$ and $z-r(s)$ are close to 0. Thus, by a comment at the very beginning of Section~3 of \cite{b}, the conditions of Theorem~1 of that paper are satisfied, and hence the central limit theorem holds (with centering by $\E A_n\sim n/2$ and scaling by $\sqrt{\operatorname{var}(A_n)}\sim\sqrt{n/12}$ as implied by \eqref{exp_var_diag}). Alternatively,  we see that
\[r(0)=1,\quad r'(0)=-\frac12,\quad\mbox{and}\quad r''(0)=\frac16,\]
so that  by \cite[Theorem~1]{b}  $\E A_n\sim\frac n2$ and $\operatorname{var}(A_n)\sim\frac n{12}$ which conforms to what we have found in \eqref{exp_var_diag}. In any event, \eqref{clt_diag} follows.

\section{Conclusion and further remarks}

In this paper we have developed a probabilistic approach to the analysis of properties of random staircase tableaux. Using this approach we established the asymptotic normality of several parameters associated with appearances of Greek letters $\alpha$, $\beta$, $\gamma$, and $\delta$ in a randomly chosen  tableau. We certainly hope that this approach will be useful in the analysis of other properties of staircase tableaux. From the combinatorial point of view, it would be of interest to analyse the number of appearances of the letter q in such tableaux. It is not clear at this point, that the method we develop is adequate to address that question, and if so, how difficult it would be to achieve. This is probably an issue worth resolving in the future. The fact that our method can be used to give new and rather complete results concerning  Greek letters and the fact that it gives an easy way of enumerating of staircase tableaux of a given size makes us cautiously optimistic.

It is perhaps worth making the following point. In our arguments we relied on observations like \eqref{eq:rplusdel} and symmetries of staircase tableaux
 to minimize the amount of work.  We wish to  emphasize, however,  that our probabilistic approach does provide a unified and systematic way of analyzing each of the statistics we considered  in a self--contained manner (i.e. not relying on relations between various parameters).  In fact, a direct proof that $\G_n$ satisfies \eqref{exact_D}, \eqref{exp_var_D}, and \eqref{dist_D} was given in \cite{dh}.  This is partly a reason we believe that our approach  has  potential of being useful in addressing other questions concerning statistics of staircase tableaux. To reiterate that point we wish  to briefly sketch a proof of \eqref{eq_diag} in Theorem~\ref{thm:diag} not relying on symmetries of the parameters.  

\subsection{The number of $\beta/\delta$ on the diagonal}
We show  that $A_n$ and $B_n$ have identical probability generating functions which, of course, implies \eqref{eq_diag}. Write $B_n=\sum_{j=1}^nb_j$, where $b_j$ is $1$ if we put a $\b$ or $\delta$ in the $j$th box on the diagonal (counting SW from the top) and is $0$ otherwise. 
We first derive the expression for the conditional probability generating function:  For $z,t$ complex,
\begin{eqnarray*}\E(t^{b_n}z^{r_n}|\mathcal D_{n-1})&=&\E(t^{b_n}z^{r_n}I_{b_n=1}|\mathcal D_{n-1})
+\E(t^{b_n}z^{r_n}I_{b_n=0}|\mathcal D_{n-1})\\&
=&t\E(z^{r_n}I_{b_n=1}|\mathcal D_{n-1})+z^{r_{n-1}+1}P(b_n=0|\mathcal D_{n-1})\\&
=&t\E(z^{r_n}|\mathcal D_{n-1})-
t\E(z^{r_n}I_{b_n=0}|\mathcal D_{n-1})+z^{r_{n-1}+1}\frac1{2\cdot3^{r_{n-1}}}
\\&=&
t\frac{z+1}2\left(\frac{z+2}3\right)^{r_{n-1}}
-
tz^{r_{n-1}}\frac1{2\cdot3^{r_{n-1}}}
+\frac z2\left(\frac z3\right)^{r_{n-1}}
\\&=&
t\frac{z+1}2\left(\frac{z+2}3\right)^{r_{n-1}}+
(1-t)\frac z2\left(\frac z3\right)^{r_{n-1}}
.
\end{eqnarray*}
This gives an expression for the bivariate probability generating function for $r_n$ and $B_n$ and a recurrence for the coefficients: 
\begin{eqnarray*}\E_nt^{B_n}z^{r_n}&=&\E_nt^{B_{n-1}}\left\{t\frac{z+1}2\left(\frac{z+2}3\right)^{r_{n-1}}+
(1-t)\frac z2\left(\frac z3\right)^{r_{n-1}}\right\}
\\&=&\frac1{2n}
\left\{t(z+1)\E_{n-1}t^{B_{n-1}}(z+2)^{r_{n-1}}+
(1-t)zE_{n-1}t^{B_{n-1}}z^{r_{n-1}}\right\}\\&=&
\frac1{2n2(n-1)\cdot\dots\cdot2(n-m+1)}\left\{\sum_{\ell=0}^mb_{m,\ell}\E_{n-m}t^{B_{n-m}}(z+2\ell)^{r_{n-m}}\right\}.
\end{eqnarray*}
Now,
\begin{eqnarray*}\E_{n-m}t^{B_{n-m}}(z+2\ell)^{r_{n-m}}&=&\frac1{2(n-m)}\Big\{t(z+2\ell+1)\E_{n-m-1}t^{B_{n-m-1}}(z+2\ell+1)^{r_{n-m-1}}\\&&\quad+(1-t)(z+2\ell)\E_{n-m-1}t^{B_{n-m-1}}(z+2\ell)^{r_{n-m-1}}\Big\},
\end{eqnarray*}
so that 
\begin{eqnarray*}b_{m,\ell}\E_{n-m}t^{B_{n-m}}(z+2\ell)^{r_{n-m}}&=&t(z+2\ell+1)
b_{m,\ell}\E_{n-m-1}t^{B_{n-m-1}}(z+2(\ell+1))^{r_{n-m-1}}\\&&\quad+(1-t)(z+2\ell)b_{m,\ell}E_{n-m-1}t^{B_{n-m-1}}(z+2\ell)^{r_{n-m-1}},
\end{eqnarray*}
which means that the coefficient in front of 
\[\E_{n-(m+1)}t^{B_{n-(m+1)}}(z+2\ell)^{r_{n-(m+1)}}\]
is
\[(1-t)(z+2\ell)b_{m,\ell}+t(z+2(\ell-1)+1)b_{m,\ell-1}.
\]
So, if we write $b_{m,\ell}=a_{m,\ell}t^\ell(1-t)^{m-\ell}$ we see that the coefficients $a_{m,\ell}=a_{m,\ell}(z)$ satisfy the  recurrence
\[a_{m+1,\ell}=(z+2\ell)a_{m,\ell}+(z+2\ell-1)a_{m,\ell-1},\]
with the initial condition 
\[a_{0,0}=1,\quad\mbox{and}\quad a_{m,\ell}=0,\quad\mbox{for}\quad \ell>m.
\]
 This is exactly the same recurrence as for the sequence $(c_{n,k})$ defined in the last section and thus 
\begin{equation}\label{gf_B}\E_nt^{B_n}z^{r_n}=\frac1{2^{n}n!}\sum_{k=0}^{n
}c_{n,k}t^k(1-t)^{n-k}
.
\end{equation}
In particular, putting $z=1$, differentiating with respect to $t$, evaluating at $t=1$,  and using $c_{n,n}=2^nn!$ and $c_{n,n-1}=2^{n-1}n!n$ we confirm that 
\[\E_nB_n=\frac1{2^nn!}(-c_{n,n-1}+nc_{n,n})=\frac1{2^nn!}(-2^{n-1}n!n+n2^nn!)=\frac n2.\]
Finally to see that $B_n$ and $A_n$ have, in fact,  the same distribution, put $z=1$ in \eqref{gf_B} and expand $(1-t)^{n-k}$ to get
\begin{eqnarray*}\sum_{k=0}^nc_{n,k}t^k(1-t)^{n-k}&=&\sum_{k=0}^nc_{n,k}t^k\sum_{j=0}^{n-k}{n-k\choose j}(-1)^{n-k-j}t^{n-k-j}\\&=&
\sum_{k=0}^n\sum_{j=0}^{n-k}t^{n-j}c_{n,k}{n-k\choose j}(-1)^{n-k-j}\\&=&
\sum_{m=0}^nt^m\left\{\sum_{k=0}^mc_{n,k}{n-k\choose n-m}(-1)^{m-k}\right\}
\\&=&
\sum_{m=0}^nt^m\left\{\sum_{k=0}^m2^kk!W(n,k){n-k\choose n-m}(-1)^{m-k}\right\}.\end{eqnarray*}
Now recall that by symmetry
\begin{eqnarray*}V(n,m)&=&V(n,n-m)=\sum_{k=0}^{n-(n-m)}2^kk!W(n,k){n-k\choose n-m}(-1)^{n-k-(n-m)}
\\&=&\sum_{k=0}^m2^kk!W(n,k){n-k\choose n-m}(-1)^{m-k},
\end{eqnarray*}
so that 
\[\E_nt^{B_n}=\frac1{2^nn!}\sum_{k=0}^nc_{n,k}t^k(1-t)^{n-k}=\frac1{2^nn!}\sum_{m=0}^nV(n,m)t^m,\]
which is exactly the same as the generating function of $A_n$ as claimed. 


\end{document}